\documentclass[12pt]{amsart}
\usepackage{amssymb}
\newtheorem{theorem}{Theorem}[section]
\newtheorem{lemma}[theorem]{Lemma}

\theoremstyle{definition}

\newtheorem{examp}[theorem]{Example}

\theoremstyle{remark}

\numberwithin{equation}{section}

\newenvironment{example}{\begin{examp}\rm}{\diams\end{examp}}
\newcommand{\diams}{\unskip\nobreak\hfil\penalty50%
\hskip1em\hbox{}\nobreak\hfil%
$\diamondsuit$\parfillskip=0pt\finalhyphendemerits=0}

\newcommand{\sn}{\par\smallskip\noindent}

\newcommand{\bn}{\par\bigskip\noindent}
\newcommand{\pars}{\par\smallskip}
\newcommand{\parm}{\par\medskip}
\newcommand{\parb}{\par\bigskip}

\newcommand{\chara}{\mbox{\rm char}\,}
\newcommand{\trdeg}{\mbox{\rm trdeg}\,}

\newcommand{\F}{\mathbb F}
\newcommand{\Fp}{\F_p}
\begin{document}

\title[Artin-Schreier defect extensions]{Correction and notes to the paper ``A classification of
Artin-Schreier defect extensions and characterizations of defectless fields''}

\author{Franz-Viktor~Kuhlmann}
\thanks{The work on this article was partially supported by a Polish Opus grant 2017/25/B/ST1/01815.\\
The author wishes to thank Anna Blaszczok for useful suggestions and careful proofreading.}
\address{Institute of Mathematics, University of Szczecin, ul.~Wielkopolska 15
70-451 Szczecin, Poland}
\email{fvk@usz.edu.pl}
%\thanks{}

\date{December 5, 2018}

\begin{abstract}
We correct a mistake in a lemma in the paper cited in the title and show that it did not affect any of the other results of the paper. To this end we prove results on linearly disjoint field extensions that do not seem to be commonly known. We give an example to show that a separability assumption in one of these results cannot be dropped (doing so had led to the mistake). Further, we discuss recent generalizations of the original classification of defect extensions.
\end{abstract}

\maketitle

%
%ÄÄÄÄÄÄÄÄÄÄÄÄÄÄÄÄÄÄÄÄÄÄÄÄÄÄÄÄÄÄÄÄÄÄÄÄÄÄÄÄÄÄÄÄÄÄÄÄÄÄÄÄÄÄÄÄÄÄÄÄÄÄÄÄÄÄÄÄÄÄÄ
%
\section{Introduction}
In the paper \cite{[Ku6]} the author introduced a classification of Artin-Schreier defect extensions. Defect extensions 
of a valued field $(K,v)$ can only appear when the characteristic of the residue field $Kv$ is positive. They constitute 
a major obstacle to the solution of the following open problems in positive characteristic:
\sn
1) local uniformization (the local form of resolution of singularities) in arbitrary dimension,
\sn
2) decidability of the field $\F_q((t))$ of Laurent series over any finite field $\F_q$, and of its perfect hull.
\sn
Both problems are closely connected with the structure theory of valued function fields of positive characteristic $p$.

Since the classification was introduced, several indications have been found that one of the two types of defects is not 
as harmful as the other. But in \cite{[Ku6]} it was only introduced for valued fields $(K,v)$ of equal positive 
characteristic (i.e., $\chara K=\chara Kv=p>0$). Recently, it was extended in \cite{[BK2]} to all defect extensions of 
prime degree, including the case of valued fields $(K,v)$ of mixed characteristic (i.e., $\chara K=0$, $\chara Kv=p>0$).
In the process of generalizing results to the mixed characteristic case (see Section 4 of \cite{[BK2]}), a mistake was 
found in the proof of Lemma~4.12 of \cite{[Ku6]}. The following claim had been stated without a reference (with a 
slightly different notation):\sn
Claim: {\it If a field $K$ is relatively algebraically closed in an extension field $F$ and $L$ is an algebraic extension 
of $K$ linearly disjoint from $F$ over $K$, then $L$ is relatively algebraically closed in $L.F$.}\sn
(Here, the compositum $L.F$ of $L$ and $F$ is taken in a fixed algebraic closure of $F$). But this claim is not 
true in general if $L|K$ is not separable. We show this by Example~\ref{ex} below. It is worth mentioning that this 
example was implicitly used by F.~Delon in \cite{D} to show that an algebraically maximal valued field is not 
necessarily defectless; this is worked out in detail in Example~3.25 of \cite{Kdef}. For the definitions of these notions and others used but not explained in these notes, and for further background, see \cite{Kdef,[Ku6],[BK2]}.

A correct version of the above claim reads as follows:\sn
{\it If a field $K$ is relatively separable-algebraically closed in an extension field $F$ and $L$ is an algebraic 
extension of $K$, then $L$ is relatively separable-algebraically closed in $L.F$.}\sn
We prove this assertion in Lemma~\ref{racsimp} in Section~\ref{sectl}. We prove more than this, in order to clarify the 
situation, but also because these results are hard to find in the literature.

\pars
In Section~\ref{sectl17} we state and prove a corrected version of the faulty Lemma~4.12. Its statement is slightly 
weaker than in the original version, as we only obtain that $K$ is relatively separable-algebraically closed in $M_w\,$.
But this suffices for the proof of the crucial Proposition~4.13 of \cite{[Ku6]}. 

\parm
Finally, let us mention that one purpose of introducing the classification of defect extensions was to prove 
Theorem~1.2 of \cite{[Ku6]}, which states: \sn
{\it A valued field of positive characteristic is henselian and defectless
if and only if it is separable-algebraically maximal and inseparably defectless.}\sn
This fact in turn was used in \cite{[K1]} to construct henselian defectless fields for a crucial example. It was hoped 
that the generalization of the classification to the mixed characteristic case would result in the proof of some analogue 
of Theorem~1.2 for this case. Unfortunately, so far we were only able to prove a partial analogue (Theorem~1.7 of
\cite{[BK2]}). The problem is that it is still not entirely clear what the analogue of purely inseparable defect 
extensions may be in mixed characteristic.

\bn
\section{A lemma about linearly disjoint extensions of fields}                           \label{sectl}
\begin{lemma}                               \label{racsimp}
Let $F|K$ be an arbitrary field extension and $L|K$ an algebraic extension.
\sn
1) Assume that $K$ is relatively algebraically closed in $F$ and $b$ is algebraic over $K$, or that 
$K$ is relatively separable-algebraically closed in $F$ and $b$ is separable-algebraic over $K$. Then $F$ and $K(b)$ 
are linearly disjoint over $K$.
\sn
2) Assume that $K$ is relatively separable-algebraically closed in $F$ and $L|K$ is separable-algebraic. Then $F$ and $L$
are linearly disjoint over $K$.
\sn
3) Assume that $K$ is relatively algebraically closed in $F$ and $L|K$ is separable-algebraic. Then $L$ is relatively 
algebraically closed in $L.F\,$.
\sn
4) Assume that $K$ is relatively separable-algebraically closed in $F$ and $L|K$ is algebraic. Then $L$ is relatively 
separable-algebraically closed in $L.F\,$.
\end{lemma}
\begin{proof}
1): Take an algebraic extension $K(b)|K$. The minimal polynomial $f\in F[X]$ of $b$ over $F$ is a divisor of the
minimal polynomial of $b$ over $K$, so all roots of $f$ are algebraic over $K$ and so are the coefficients of $f$
since they are symmetric functions in these roots. If $K$ is assumed to be relatively algebraically closed in $F$, it 
follows that $f\in K[X]$. If in addition $b$ is separable over $K$, then also the coefficients of $f$ are separable 
over $K$ and it suffices to assume that $K$ is relatively separable-algebraically closed in $F$ to obtain that 
$f\in K[X]$. In both cases, $f$ is also the minimal polynomial of $b$ over $K$. Thus,
$[F(b):F]=[K(b):K]$, showing that $F$ and $K(b)$ are linearly disjoint over $K$.
\sn
2) Now let $L|K$ be a separable-algebraic extension. Then $L$ is a union of 
simple subextensions of $L|K$; if $K$ is relatively separable-algebraically closed in $F$, then by part 1), these 
are linearly disjoint from $F$ over $K$. It then follows that $L$ itself is linearly disjoint from $F$ over $K$. 
\sn
3) Assume that $K$ is relatively algebraically closed in $F$ and $L|K$ is separable-algebraic. Then also 
$L.F|F$ is separable algebraic (since the minimal polynomial of any $b\in L$ over $F$ is a divisor of its  
minimal polynomial over $K$).

Let $a\in L.F$ be algebraic over $L$; hence, $a$ is also algebraic over $K$, and by what we have just shown, it is 
separable-algebraic over $F$. By part a), the minimal polynomial of $a$ over $F$ coincides with that over $K$, so we 
know that $a$ is separable-algebraic over $K$. Consequently, $L(a)|K$ is a separable-algebraic extension. From part 2) 
we infer that $L(a)$ is linearly disjoint from $F$ over $K$. By \cite[Chapter VIII, Proposition 3.1]{L},
$L(a)$ is linearly disjoint from $L.F$ over $L$. In particular, $a\in L.F$ implies $a\in L$. This proves that $L$ is
relatively algebraically closed in $L.F\,$.
\sn
4) Assume that $K$ is relatively separable-algebraically closed in $F$ and $L|K$ is algebraic. If $K'$ denotes the 
relative algebraic closure of $K$ in $F$, then $K'|K$ is purely inseparable, and consequently, the same is true for the 
algebraic subextension $L.K'|L$ of $L.F|L$. Therefore, if we are able to show that $L.K'$ is relatively
separable-algebraically closed in $L.F\,$, then the same holds for $L$. We may thus assume from the start that 
$K$ is relatively algebraically closed in $F$, and we need to show that $L$ is relatively 
separable-algebraically closed in $L.F\,$.

Let $L_0|K$ be the maximal separable subextension of $L|K$, so $L|L_0$ is purely inseparable. By part 3), $L_0$ is 
relatively algebraically closed in $L_0.F\,$. Suppose that $L$ is not relatively separable-algebraically closed in 
$L.F\,$. Then the relative algebraic closure of $L$ in $L.F$ contains a nontrivial separable-algebraic subextension 
$L_1$ of $L_0\,$. By part 2), $L_1$ is linearly disjoint from $L_0.F$ over $L_0\,$. This shows that $L_1.F|L_0.F$ is
a nontrivial separable subextension of $L.F|L_0.F\,$. But as $L|L_0$ is purely inseparable, so is $L.F|L_0.F\,$. This contradiction shows that $L$ is relatively separable-algebraically closed in $L.F\,$.
\end{proof}

\parb
Assertion 2) of the lemma fails when $L|K$ is algebraic but neither separable nor simple, even when $K$ is relatively 
algebraically closed in $F$. Likewise, assertion 3) fails
when $L|K$ is algebraic but not separable. This will be shown in the following example.
\begin{example}                                     \label{ex}
We take elements $t,x,y$ which are algebraically independent over $\Fp$. We choose any prime $p$, set 
\[
F\>:=\>\Fp(t,x,y)
\]
and define
\[
s\>:=\>x^p+ty^p\;\mbox{\ \ and\ \ }\; K\>:=\>\Fp(t,s)\>.
\]
Then $s$ is transcendental over $\F_p(t)$, so $K$ has $p$-degree $2$, that is, $[K:K^p]=p^2$. 

\pars
We prove that $K$ is relatively algebraically closed in $F$. Take $b\in F$ algebraic over $K$. The
element $b^p$ is algebraic over $K$ and lies in $F^p=\Fp(t^p,x^p,y^p)$ and thus also in $K(x)= \Fp(t,x,y^p)$. Since 
$\trdeg\Fp(t,x,y^p)|\Fp=3$ while $\trdeg K|\F_p=2$, we see that $x$ is transcendental over $K$. Therefore, 
$K$ is relatively algebraically closed in $K(x)$ and thus, $b^p\in K$. Consequently, $b\in K^{1/p}= 
\Fp(t^{1/p},s^{1/p})$. Write
\[
b\;=\; r_0+r_1s^{\frac{1}{p}}+\ldots+r_{p-1}s^{\frac{p-1}{p}}
\mbox{ \ \ \ with \ } r_i\in \Fp(t^{1/p},s)=K(t^{1/p})\;.
\]
Since $s^{1/p}=x+t^{1/p}y$, we have that
\[
b\;=\; r_0 \,+\, r_1x +\ldots+ r_{p-1}x^{p-1}
\,+\ldots+\, t^{1/p}r_1y +\ldots+ t^{(p-1)/p}r_{p-1}y^{p-1}
\]
(in the middle, we have omitted the summands in which both $x$ and $y$
appear). Since $x,y$ are algebraically independent over $\Fp$, the
$p$-degree of $\F_p(x,y)$ is $2$, and the elements $x^i y^j$, $0\leq
i<p$, $0\leq j<p$, form a basis of $\Fp(x,y)|\Fp(x^p,y^p)$. Since $t^{1/p}$ is transcendental over $\Fp(x^p,y^p)$, we 
know that $\Fp(x,y)$ is linearly disjoint from $\Fp(t^{1/p},x^p,y^p)$ and hence also from $\Fp(t,x^p,y^p)$
over $\Fp(x^p,y^p)$. This shows that the elements $x^i y^j$ also form a basis of $F|\Fp(t,x^p,y^p)$ and
are still $\Fp(t^{1/p},x^p,y^p)$-linearly independent. Hence, $b$ can
also be written as a linear combination of these elements with
coefficients in $\Fp(t,x^p,y^p)$, and this must coincide with the above
$\Fp(t^{1/p},x^p,y^p)$-linear combination which represents $b$. That is,
all coefficients $r_i\,$ {\it and\/} $\,t^{i/p}r_i$, $1\leq i<p$, are in
$\Fp(t,x^p,y^p)$. Since $t^{i/p}\notin\Fp(t,x^p,y^p)$, this is
impossible unless they are zero. It follows that $b=r_0\in K(t^{1/p})$.
Assume that $b\notin K$. Then $[K(b):K]=p$ and thus, $K(b)=K(t^{1/p})$
since also $[K(t^{1/p}):K]=p$. But then $t^{1/p}\in K(b)\subset F$, a
contradiction. This proves that $K$ is relatively algebraically closed in $F$.

\pars
We show that
\[
K(t^{1/p},s^{1/p})
\]
is not linearly disjoint from $F$ over $K$. Indeed, $s^{1/p}=x+t^{1/p}y\in K(t^{1/p},x,y)$, which implies that 
\[
[K(t^{1/p},s^{1/p}).F:F]\>=\>[F(t^{1/p}):F]\>=\>p\><\>p^2\>=\>[K(t^{1/p},s^{1/p}):K]\>.
\]

Further, we see that while $K(t^{1/p})$ is linearly disjoint from $F$ over $K$, it is not relatively algebraically 
closed in $F(t^{1/p})=K(t^{1/p}).F$ since $s^{1/p}\in F(t^{1/p})\setminus K(t^{1/p})$. We have shown that the 
separability condition on $L|K$ in parts 2) and 3) of Lemma~\ref{racsimp} is necessary.
\end{example}

\bn
\section{Corrected version of Lemma~4.12}                           \label{sectl17}
We consider a valued field $(K_0,v)$.
\begin{lemma}                            \label{l7}
Assume that for every coarsening $w$ of $v$ (including $v$ itself),
$K_0$ admits a maximal immediate extension $(N_w|K_0,w)$ such that $K_0$ is relatively separable-algebraically
closed in $N_w$. If $(K|K_0,v)$ is a finite and defectless extension,
then for every coarsening $w$ of~$v$ (including $v$ itself), $(M_w,w) =
(N_w.K,w)$ is a maximal immediate extension of $(K,w)$ such that $K$ is
relatively separable-algebraically closed in $M_w$.
\end{lemma}
\begin{proof}
Since $(K|K_0,v)$ is defectless by hypothesis, the same is true for the
extension $(K|K_0,w)$ by Lemma~2.4 of \cite{[Ku6]}. We note that $(K_0,w)$ is
henselian since it is assumed to be separable-algebraically closed in
the henselian field $(N_w,w)$. So we may apply Lemma~2.5 of \cite{[Ku6]}: since
$(N_w|K_0,w)$ is immediate and $(K|K_0,w)$ is defectless, $(N_w.K|K,w)$ is immediate. 
%and $N_w$ is linearly disjoint from $K$ over $K_0$. The latter shows that 
By part 4) of Lemma~\ref{racsimp}, $K$ is relatively separable-algebraically closed in $N_w.K$. On the
other hand, $(M_w,w)=(N_w.K,w)$ is a maximal field, being a finite extension of a maximal field.
\end{proof}

\end{document}